\documentclass{amsart}
\usepackage{art_styles1}

\begin{document}

\title[Embeddability of countably branching trees]{On the embeddability of the family of countably branching trees into quasi-reflexive Banach spaces}

\author{Y.~Perreau}
\address{Yoël Perreau, Laboratoire de Math\'ematiques de Besan\c con, Universit\'e Bourgogne Franche-Comt\'e, CNRS UMR-6623, 16 route de Gray, 25030 Besan\c con C\'edex, Besan\c con, France}
\email{yoel.perreau@univ-fcomte.fr}

\thanks{The author is supported by the French ``Investissements d'Avenir'' program, project ISITE-BFC (contract
 ANR-15-IDEX-03).}
\keywords{}
\subjclass[2010]{46B20, 46B80, 46B85, 46T99}

\maketitle

\begin{abstract}

In this note we extend to the quasi-reflexive setting the result of F. Baudier, N. Kalton and G. Lancien concerning the non-embeddability of the family of countably branching trees into reflexive Banach spaces whose Szlenk index and Szlenk index from the dual are both equal to the first infinite ordinal $\omega$.  In particular we show that the family of countably branching trees does neither embed into the James space $\mathcal{J}_p$ nor into its dual space $\mathcal{J}_p^*$ for $p\in(1,\infty)$.

\end{abstract}

\section{Introduction}

Let $(T_N)_{N\geq 1}$ be the family of countably branching trees endowed with the hyperbolic distance. The main result in the present paper is the following theorem.

\begin{thm}

Let $X$ be a quasi-reflexive Banach space. If the Szlenk index $S_Z(X)$ of the space $X$ satisfies $S_Z(X)\leq\omega$ and if the Szlenk index $S_Z(X^*)$ of its dual space  satisfies $S_Z(X^*)\leq \omega$, then the family $(T_N)_{N\geq 1}$ does not equi-Lipschitz embed into $X$.

\end{thm}

Let us briefly recall the context and the motivation of this theorem. In 1986, J. Bourgain gave in his paper \cite{Bourgain} a metric invariant characterizing super-reflexivity: the non equi-Lipschitz embeddability of the family $(D_N)_{N\geq 1}$ of dyadic trees endowed with the hyperbolic distance. His result is the following.

\begin{thm}

Let $X$ be a Banach space. Then $X$ is super-reflexive if and only if the family $(D_N)_{N\geq 1}$  does not equi-Lipschitz embed into $X$.

\end{thm}

This was the first step in the so called \emph{Ribe program}  which looks for metric invariants characterizing \emph{local properties} of Banach spaces. The reader can have a look at \cite{Naor} for a detailed introduction to the Ribe program and for a survey of results in this direction. A short proof of the non-embeddability of the family of dyadic trees into a super-reflexive space was given more recently by R. Kloeckner in \cite{Kloeckner} using uniform convexity and a self-improvement argument.

In \cite{BKL}, F. Baudier, N. J. Kalton and G. Lancien introduced a new metric invariant in order to give a metric characterization of \emph{asymptotic properties} of Banach spaces: the non equi-Lipschitz embeddability of the family $(T_N)_{N\geq 1}$ of countably branching trees endowed with the hyperbolic distance. The main tool in their paper is a derivation index called Slzenk index. We will introduce these objects in section $2$. They proved the following results.

\begin{thm}

Let $X$ be a separable Banach space. If $S_Z(X)>\omega$ or if $S_Z(X^*)>\omega$, then the family $(T_N)_{N\geq 1}$ equi-Lipschitz embeds into $X$.

\end{thm}

\begin{thm}

Let $X$ be a reflexive separable Banach space. If $S_Z(X)\leq \omega$ and $S_Z(X^*)\leq \omega$, then the family $(T_N)_{N\geq 1}$ does not equi-Lipschitz embed into $X$.

\end{thm}

Note that the assumption of separability can be removed in both theorems by using properties of the Szlenk index. Using an argument à la Kloeckner and the property $(\beta)$ of Rolewicz, F. Baudier and S. Zhang gave in \cite{Baudier_Zang} a shorter proof of the second theorem. However, this argument cannot be extended to a more general setting since $(\beta)$ implies reflexivity.

In \cite{Kalton}, N. J. Kalton applied results coming from the study of Orlicz sequence spaces to get estimates on the spreading models of Banach spaces which coarse-Lipschitz embed into asymptotically uniformly convex spaces. Inspired by this method, we will give in section $4$ estimates on certain trees in the bi-dual space of quasi-reflexive Banach spaces enjoying asymptotic properties. These estimates will be key tools to extend the proof from \cite{BKL} of the non-embeddability theorem to the quasi-reflexive setting. This is done in section $3$. 

Using this result, we answer briefly in section $5$ the question of the embeddability of the family of countably branching trees into James spaces. Namely, we have the following.

\begin{thm}

Let $p\in(1,\infty)$. The family $(T_N)_{N\geq 1}$ does not equi-Lipschitz embed in the James space $\mathcal{J}_p$ and it does not equi-Lipschitz embed in its dual $\mathcal{J}_p^*$.

\end{thm}

\section{Definitions, notation and preliminary results}

Let us first recall some definitions. Let $(M,d)$ and $(N,\delta)$ be two metric spaces. We say that $M$  Lipschitz embeds into $N$ if there exist a map $f:M\to N$ and constants $a,b>0$ such that $$\forall x,y\in M, \ ad(x,y)\leq \delta(f(x),f(y))\leq bd(x,y).$$ Furthermore let $(M_i,d_i)_{i\in I}$ be a family of metric spaces. We say that $(M_i)_{i\in I}$ equi-Lipschitz embeds into $N$ if there is a family of maps $(f_i:M_i\to N)_{i\in I}$ and constants $a,b>0$ such that $$\forall i\in I,\ \forall x,y\in M_i, \ ad_i(x,y)\leq \delta(f_i(x),f_i(y))\leq bd_i(x,y).$$

A Banach space $X$ is quasi-reflexive if the quotient $X^{**}/X$ is of finite dimension or equivalently if there is a finite dimensional space $E$ such that $X^{**}=X\oplus E$.

For every $N\geq 1$ let $T_N=\{\emptyset\}\cup\bigcup_{n=1}^N\N^n$.  If $s=(s_1,\dots, s_n)\in T_N$ for some $1\leq n\leq N$, we denote $s_{\lvert k}=(s_1,\dots s_k)$ for every $1\leq k\leq n$ and $s_{\lvert 0}=\emptyset$. Also, we write $s^-=s_{\lvert n-1}$ the predecessor of $s$. There is a natural graph structure on $T_N$ obtained by putting edges between each non-empty sequence $s$ and its predecessor $s^-$. Equipped with the induced hyperbolic or graph distance $d$, $T_N$ is called countably branching tree with $N$ steps. If $s,t\in T_N$, we write $s\leq t$ whenever the sequence $t$ is an extension of the sequence $s$. This defines an ordering  on $T_N$ and allows us to introduce the greatest common ancestor of $s$ and $t$ denoted $a_{s,t}$. For every $s\in T_N$, let $\abs{s}$ be the length of the sequence $s$. The distance $d$ on $T_N$ is also defined by the formula $$d(s,t)=d(a_{s,t},s)+d(a_{s,t},t)=\abs{s}+\abs{t}-2\abs{a_{s,t}}.$$ Finally, if $s=(s_1,\dots s_n)$ and $t=(t_1,\dots t_m)$, let $s\smallfrown t=(s_1,\dots s_n,t_1,\dots, t_m)$ and let $s\propto t=(s_1,t_1,\dots, s_n,t_n)$ if $m=n$. We say that a subset $T$ of $T_N$ is a full subtree of $T_N$ if $\emptyset\in T$ and if every sequence $t$ in $T$ of length at most $N-1$ has an infinite number of direct successors in $T$. That is to say $T\cap \edtq{t\smallfrown n}{n\in \N}$ is infinite.

For every infinite subset $\M$ of $\N$ and for every $k\geq 1$, let $$[\M]^k=\edtq{(n_1,\dots, n_k)}{n_1<\dots<n_k, \ n_i\in \M}.$$ Let us mention two versions of Ramsey's theorem we will need in the sequel (see for example \cite{Gowers}). 

\begin{thm}

Let $k\geq 1$ and $A\subset [\N]^k$. Assume that for every infinite subset $\M$ of $\N$, the set $[\M]^k$ has a non-empty intersection with $A$. Then we can find a infinite subset $\M$ of $\N$ such that $[\M]^k\subset A$.

\end{thm}

\begin{thm} 

Let $(K,\rho)$ be a compact metric space and let $k\geq 1$. For every map $f:[\N]^k\to K$ and for every $\varepsilon>0$, there is an infinite subset $\M$ of $\N$ such that $\diam\ f\left( [\M]^k\right) \leq \varepsilon$.

\end{thm}

Let $X$ be a Banach space, let $K$ be a weak$^*$-compact subset of the dual space $X^*$ and fix $\varepsilon>0$.  Denote by $\mathcal{V}$ the set of all weak$^*$-open subsets $V$ of $K$ satisfying $\diam\ V \leq \varepsilon$ and let $s_\varepsilon(K)=K\backslash\left(\bigcup_{V\in \mathcal{V}}V\right)$. We define inductively subsets $s_\varepsilon^\alpha(K)$ of  $K$ for every ordinal $\alpha$ by $s_\varepsilon^1(K)=s_\varepsilon(K)$, $s_\varepsilon^{\alpha+1}(K)=s_\varepsilon(s_\varepsilon^{\alpha}(K))$ if $\alpha\geq 1$ and $s_\varepsilon^{\alpha}=\bigcap_{\beta<\alpha}s_\varepsilon^{\beta}$ if $\alpha$ is a limit ordinal. Then let $S_Z(K,\varepsilon)=\inf\edtq{\alpha}{s_\varepsilon^{\alpha}(K)=\emptyset}$ if such an $\alpha$ exists and let $S_Z(K,\varepsilon)=\infty$ otherwise.  Finally let $S_Z(K)=\sup_{\varepsilon>0} S_Z(K,\varepsilon)$. The Szlenk index of $X$ is $S_Z(X)=S_Z\left(B_{X^*}\right)$ where $B_{X^*}$ denotes the closed unit ball of $X^*$. It was introduced in a different form by W. Szlenk in \cite{Szlenk} in order to prove that there is no separable reflexive universal Banach space for the class of separable reflexive Banach spaces. An extensive study of the properties  and applications of the Slzenk index can be found in \cite{Lancien_survey_Szlenk}.

Our main tools are asymptotic uniform properties of norms. Following V. Milman \cite{Milman}, we introduce two moduli: for all $t\geq 0$, let $$\modlissunif{X}{t}=\sup_{x\in S_X} \inf_{Y}\sup_{y\in S_Y}(\norme{x+ty}-1)$$ where $Y$ runs through all closed linear subspaces of $X$ of finite co-dimension (and $S_X$ denotes the unit sphere of $X$) and $$\modconvunifw{X}{t}=\inf_{x\in S_{X^*}}\sup_{E}\inf_{\underset{\norme{y^*}=1}{y^*\in E}} (\norme{x^*+ty^*}-1)$$ where $E$ runs through all weak$^*$-closed subspaces of $X^*$ of finite co-dimension. We say that $\norme{.}_X$ is asymptotically uniformly smooth (in short AUS) if $\lim_{t\to 0}\frac{\modlissunif{X}{t}}{t}=0$ and we say that $\norme{.}_{X^*}$ is weak$^*$ asymptotically uniformly convex (in short AUC$^*$) if $\modconvunifw{X}{t}>0$ for all $t>0$. If $p\in(1,\infty)$, we say that $\norme{.}_X$ is $p$-AUS (respectively $\norme{.}_{X^*}$ is $p$-AUC$^*$) if there is a constant $c>0$ such that for all $t\in[0,1]$, $\modlissunif{X}{t}\leq ct^p$ (respectively $\modconvunifw{X}{t}\geq ct^p$). There is a nice duality result concerning these moduli (see for example \cite{DKLR}).

\begin{thm}

Let $X$ be a Banach space and let $p,q\in(1,\infty)$ be conjugate exponents. Then $\norme{.}_X$ is $p$-AUS if and only if $\norme{.}_{X^*}$ is $q$-AUC$^*$.

\end{thm}

In \cite{KOS}, the following renorming theorem was proved.

\begin{thm}\label{renorming}

Let $X$ be a separable space. The following assertions are equivalent.

\begin{enumerate}

\item{ The space $X$ admits an equivalent AUS norm. }

\item{ The space $X$ admits an equivalent $p$-AUS norm for some $p\in (1,\infty)$. }

\item{ The space $X$ satisfies $S_Z(X)\leq \omega$.}

\end{enumerate}

\end{thm}

Now, let us introduce some inequalities using the asymptotic uniform moduli we introduced above.  Let $X$ be a Banach space. First, we have a well known result concerning $\overline{\delta}^*_{X}$.

\begin{prp}

Fix $x^*\in S_{X^*},\ \sigma\geq 0$ and  $\varepsilon>0$. There is a weak$^*$-neighborhood $V$ of $0$ such that $$\forall y^*\in V,\ \norme{y^*}\geq \sigma\implies \norme{x^*+y^*}\geq 1+\modconvunifw{X}{\sigma}-\varepsilon.$$

\end{prp}

Using this, we deduce the following.

\begin{lm}\label{lm_w_conv}

Fix $x^*\in S_{X^*},\ R\geq 0$ and  $\varepsilon>0$. There is a weak$^*$-neighborhood $V$ of $0$ such that $$\forall y^*\in V\cap RB_{X^*}, \ \norme{x^*+y^*}\geq 1+\modconvunifw{X}{\norme{y^*}}-\varepsilon.$$

\end{lm}

\begin{proof}

Fix $\eta\in(0,1)$ and take a finite $\eta$-net $(\sigma_i)_{1\leq i\leq n}$ in $[0,R]$ containing $0$. Applying the preceding proposition we get a weak$^*$-neighborhood $V$ of $0$ such that $$\forall 1\leq i\leq n,\ \forall y^*\in V,\ \norme{y^*}\geq \sigma_i\implies \norme{x^*+y^*}\geq 1+\modconvunifw{X}{\sigma_i}-\eta.$$

Now take $y^*\in V$ with $\norme{y^*}\leq R$. We can find some $1\leq i_0\leq n$ such that $\sigma_{i_0}\leq \norme{y^*}\leq\sigma_{i_0}+\eta$. Applying the preceding inequality we get \begin{align*}
\norme{x^*+y^*} & \geq 1+\modconvunifw{X}{\norme{y^*}}+\modconvunifw{X}{\sigma_{i_0}}-\modconvunifw{X}{\norme{y^*}}-\eta \\
&\geq 1+\modconvunifw{X}{\norme{y^*}}-\modcont{\overline{\delta}^*_X}{\eta}-\eta
\end{align*} where $\modcont{\overline{\delta}^*_X}{.}$  is the modulus of continuity of the function $\overline{\delta}^*_X$. The result follows because $\overline{\delta}^*_X$ is uniformly continuous.

\end{proof}

Similar results exist for $\overline{\rho}_X$ (with weak-neighborhoods of $0$) but in fact, we can do a bit  better. The following improvement was given in \cite{LR}.

\begin{prp}

Fix $x\in S_X,\ \sigma\geq 0$ and $\varepsilon>0$. There is a weak$^*$-neighborhood $V$ of $0$ in $X^{**}$ such that $$\forall y^{**}\in V,\ \norme{y^{**}}\leq \sigma; \ \norme{x+y^{**}}\leq 1+\modlissunif{X}{\sigma}+\varepsilon.$$

\end{prp}

In the same way as before, we deduce the following.

\begin{lm}\label{lm_smooth}

Fix $x\in S_X,\ R\geq 0$ and  $\varepsilon>0$. There is a weak$^*$-neighborhood $V$ of $0$ in $X^{**}$ such that $$\forall y^{**}\in V\cap RB_{X^{**}}, \ \norme{x+y^{**}}\leq 1+\modlissunif{X}{\norme{y^{**}}}+\varepsilon.$$

\end{lm}

In particular, we will consider standard weak$^*$-neighborhoods of $0$ in $X^{**}$ of the form $$V_{x^*_1,\dots,x^*_m;\varepsilon}=\edtq{x^{**}\in X^{**}}{\forall 1\leq i\leq m,\ \abs{x^{**}(x^*_i)}<\varepsilon}$$ with $x^*_1,\dots,x^*_k\in X^*$ and $\varepsilon>0$

To finish this section, let us give some results coming from \cite{Kalton} using the theory of Orlicz sequence spaces. Let us assume that our space $X$ is infinite dimensional and that it does not contain $c_0$. Then it is known that the function $\overline{\rho}_X$ satisfies the condition $\overline{\rho}_X(t)>0$ for all $t>0$ (see for example \cite{JLPS} section 2 for more information and for references). In this case, the modulus $\overline{\rho}_X$ is an Orlicz function, which is to say it is a continuous non-decreasing and convex function satisfying $\overline{\rho}_X(0)=0$ and $\overline{\rho}_X(t)>0$ for every $t>0$. We define by induction functions $N_k^{\overline{\rho}_X}$ on $\R^k$ first by $N_2^{\overline{\rho}_X}(x,y)=\abs{x}\left(1+\overline{\rho}_X\left(\frac{\abs{y}}{\abs{x}}\right)\right)$ if $x\neq 0$ and $N_2^{\overline{\rho}_X}(0,y)=\abs{y}$ and then by $N_k^{\overline{\rho}_X}(x_1,\dots,x_k)=N_2^{\overline{\rho}_X}(N_{k-1}^{\overline{\rho}_X}(x_1,\dots,x_{k-1}),x_k)$ if $k\geq 3$. These functions define norms on $\R^k$. Now let us assume that the dual space $X^*$ is $AUC^*$. Then the function $\overline{\delta}^*_{X}$ satisfies the condition $\overline{\delta}^*_{X}(t)>0$ for all $t>0$. Since $\overline{\delta}^*_{X}$  need not be an Orlicz function, we introduce an auxiliary  function $\delta(t)=\int_0^t \frac{\modconvunifw{X}{s}}{s} ds$ which happens to be an Orlicz function and to satisfy $\frac{1}{2}\overline{\delta}^*_{X}\leq \delta\leq \overline{\delta}^*_{X}$. We define as before the norms $N_k^\delta$ on $\R^k$. These norms satisfy the following properties.

\begin{lm}\label{Orlicz_Inequalities}

Let $X$ be an infinite dimensional Banach space.

\begin{enumerate}

\item{ If $X$ is $p-AUS$ for some $p\in (1,\infty)$ and if it does not contain $c_0$, there is a constant $A>0$ such that: $$\forall k\geq 1,\ \forall v\in \R^k,\ N_k^{\overline{\rho}_X}(v)\leq A\norme{v}_{l_p^k}.$$ }

\item{ If $X^*$ is $q-AUC^*$ for some $q\in (1,\infty)$, there is a constant $a>0$ such that: $$\forall k\geq 1,\ \forall v\in \R^k,\ N_k^\delta(v)\geq a\norme{v}_{l_q^k}.$$ }

\end{enumerate}

\end{lm}

\section{Main result}

This section is devoted to the proof of the main result.

\begin{thm}\label{Main_result}

Let $X$ be a quasi-reflexive Banach space satisfying $S_Z(X)\leq\omega$ and $S_Z(X^*)\leq \omega$. Then the family $(T_N)_{N\geq 1}$ does not equi-Lipschitz embed into $X$.

\end{thm}

In order to prove our result, let us consider a quasi-reflexive infinite dimensional Banach space $X$. We suppose that the family $(T_N)_{N\geq 1}$ equi-Lipschitz embeds into $X$. We may assume that there exist a constant $c>0$ and functions  $f_N:T_N\to X$ with $f_N(\emptyset)=0$ such that $$\forall N\geq 1,\  \forall s,t\in T_N,\ d(s,t)\leq \norme{f_N(s)-f_N(t)}\leq cd(s,t).$$ Considering the closed linear span of $\bigcup_{N\geq 1}f_N(T_N)$ in $X$, we may assume that $X$ and therefore all its iterated duals are separable. 

Now suppose that $X$ satisfies $S_Z(X)\leq\omega$ and $S_Z(X^*)\leq \omega$. By the renorming theorem \ref{renorming}, we may assume that $\norme{.}_X$ is $p-AUS$ and that the dual space $X^*$ admits an equivalent $q^*-AUS$ norm $\abs{.}$ for some $p,q^*\in(1,\infty)$. As mentioned we get that the dual norm $\abs{.}$ on $X^{**}$ is $q-AUC^*$ where $q$ is the conjugate exponent of $q^*$. We may assume that $$\abs{.}\leq\norme{.}\leq e\abs{.}$$ for some constant $e>0$ on $X^{**}$. 

Let us fix some $N\geq 1$ which is to be determined later and let us write $f=f_N$. We will be considering the function $f$ as a function with values in $X^{**}$. For all $t\in T_N,\ t\neq\emptyset,$ we put $$z(t)=f(t)-f(t^-).$$ Note that $\norme{z(t)}\leq c$ for every $t\in T_N$. Therefor, using weak$^*$-sequential compactness and passing to a full subtree, we may assume that for all $1\leq j\leq N$ and for all $\forall t\in T_{N-j}$, the iterated weak$^*$-limit $$\partial_j z(t)= w^*-\lim_{n_1}\ \dots\ w^*-\lim_{n_j} z(t\smallfrown (n_1,\dots, n_j))$$ is well defined. We also denote $\partial_0 z(t)=z(t)$. Note that $\norme{\partial_j z(t)}\leq c$ by lower semi-continuity of the norm. For all $1\leq j\leq k\leq N$ and for all $t\in T_N$ of length $\abs{t}\geq j$, we introduce $$z_{k,j}(t)=\partial_{k-j} z\left(t_{\lvert j}\right)-\partial_{k-j+1}z\left(t_{\lvert j-1}\right).$$ Also, let $z_{k,0}(t)=\partial_kz(\emptyset)$. Note that $z_{k,j}(t)$ only depends on the $j$ first coordinates of the sequence $t$ and that $\norme{z_{k,j}(t)}\leq 2c$. Moreover, we have the following properties. The proof of these results is straightforward but we will apply them often in the sequel.
\begin{prp}

For all  $t\in T_N,\ t\neq\emptyset$, we have:
 \begin{enumerate}
\item{ $f(t)=\sum_{k=1}^{\abs{t}}z(t_{\lvert k})=\sum_{k=1}^{\abs{t}}\sum_{j=0}^k z_{k,j}(t)$ }
\item{ $\forall 1\leq l\leq k\leq \abs{t},\ \norme{\sum_{j=0}^l z_{k,j}(t)}\leq c$  }
\item{ $\forall 1\leq j\leq k\leq N,\ j\leq \abs{t}$, $w^*-\lim_n z_{k,j}\left(t_{\lvert j-1}\smallfrown n\right)=0.$ }
\end{enumerate}

\end{prp}

Now let us assume that $N=3Q^M$ for some $Q>3$ and $M\geq 1$. Then for all $1\leq k\leq N$, there is a unique $1\leq M_k\leq M+1$ such that $Q^{M_k-1}\leq k <Q^{M_k}$. Thus we can define exponentially decreasing  indices $\alpha_{k,0}=k$, $\alpha_{k,r}=k-Q^r$ for $1\leq r<M_k$ and $\alpha_{k,M_k}=-1$. We consider block functions $w_{k,r}$ defined on the roof of the tree $T_N$ by $$w_{k,r}(t)=\sum_{j=\alpha_{k,r}+1}^{\alpha_{k,r-1}}z_{k,j}(t).$$ Our goal in the sequel will be to give upper and lower estimates of the quantity $$\sum_{k=1}^{N}\sum_{r=1}^{M_k}\norme{w_{k,r}(t)}$$ in a certain full subtree in order to get a contradiction when $Q$ and $M$ are sufficiently big. 

In the reflexive case, it is possible to get such estimates using a result from  \cite{KOS} where the space is embedded into a Banach space admitting a finite dimensional decomposition in which nice upper and lower $\ell_p$ and $\ell_q$  estimates hold. We will replace this result in our setting by the two following propositions.

\begin{prp}\label{estimates_below}

For all $\eta>0$, there is a full subtree $T$ of $T_N$ such that for all $1\leq L\leq N,$ for all $0\leq i_1\leq j_1<\dots<i_L\leq j_L\leq N$ and for all $N\geq k_i\geq j_i$, we have  $$\abs{\sum_{l=1}^LB_l(t)} \geq a\left(\sum_{l=1}^L \abs{B_l(t)}^q\right)^{\frac{1}{q}}-\eta$$ whenever $t$ is an element of $T$ of length $\abs{t}\geq j_L$, where $B_l$ is the block function defined by $$B_l(t)= \sum_{j=i_l}^{j_l}z_{k_l,j}(t)$$ and $a$ is the constant obtained by our considerations about the  Orlicz function  $\delta$ in the second section (lemma \ref{Orlicz_Inequalities} applied to the space $(X^{**},\abs{.})$).

\end{prp}

\begin{prp}\label{estimates_above}

For all $\eta>0$, there is a full subtree $T$ of $T_{2N}$ such that for all $1\leq L\leq N$, for all $0\leq i_1\leq j_1<\dots<i_L\leq j_L\leq N$ and for all $N\geq k_i,\geq j_i$, we have $$\norme{\sum_{l=1}^L B_l(s)-B_l(t)} \leq A\left(\sum_{l=1}^L \norme{\sum_{j=i_l}^{j_l}B_l(s)-B_l(t)}^p\right)^{\frac{1}{p}}+\eta$$ whenever $s$ and $t$ are two elements of $T_N$ of length $\abs{t}=\abs{s}\geq j_L$ such that the interlaced sequence $s\propto t=(s_1,t_1,s_2,\dots)$ belongs to $T$, where $B_l$ is the block function defined by $$B_l(t)= \sum_{j=i_l}^{j_l}z_{k_l,j}(t)$$ and $A$ is the constant obtained by our considerations about the  Orlicz function  $\overline{\rho}_X$ in the second section (lemma \ref{Orlicz_Inequalities}).

\end{prp}

The proof of these two propositions will be done in the next section to make the reading lighter. We turn to the proof of the main result.

\begin{proof}[Proof of \ref{Main_result}]

Fix $\eta>0$ and assume that the two propositions are satisfied respectively on the whole $T_N$ and $T_{2N}$ for this constant. 

First, we apply \ref{estimates_below} for every $1\leq k\leq N$ to the block functions $w_{k,r}$ with $r$ running from $1$ to $M_k$. We get $$\left(\sum_{r=1}^{M_k}\abs{w_{k,r}(t)}^q\right)^{\frac{1}{q}} \leq \frac{1}{a}\left(\abs{\sum_{r=1}^{M_k}w_{k,r}(t)} +\eta\right)\leq \frac{c+\eta}{a}$$ for every $t\in T_N$ of length $N$ because $$\sum_{r=1}^{M_k}w_{k,r}(t)=\sum_{j=0}^kz_{k,j}(t)$$ is of norm at most $c$. Then assuming that $\eta\leq c$ and using Hölder's inequality, we get $$\sum_{k=1}^{N}\sum_{r=1}^{M_k}\abs{w_{k,r}(t)}\leq \frac{2c}{a}(M+1)^{\frac{1}{q^*}}N.$$ So $$\sum_{k=1}^{N}\sum_{r=1}^{M_k}\norme{w_{k,r}(t)}\leq (M+1)^{\frac{1}{q^*}}Q^{M+1}$$ if $Q$ is was chosen bigger than $\frac{6ce}{a}$.

Second, we want to get an estimate from below. To do that, we will use some computation tricks. We start with an easy lemma. 

\begin{lm}

Let $1\leq m\leq M$. For all $Q^m\leq l\leq N-Q^m$ and for all  $s,t\in T_N$ such that $\abs{s}=\abs{t}\geq l+Q^m$ and $\abs{a_{s,t}}=l$, we have $$\norme{\sum_{k=l+1}^{l+Q^m}\sum_{r=1}^m w_{k,r}(s)-w_{k,r}(t)}\geq 2Q^m.$$

\end{lm}

Note that the condition $l\geq Q^m$ is crucial in order to ensure that the $w_{k,r}$ appearing in the sums are well defined. 

\begin{proof}

Let us recall that $z_{k,j}$ only depends on the $j$ first coordinates of the sequence. So if we take $s,t\in T_N$ satisfying the properties of the lemma and if we take $1\leq j\leq l$, then $z_{k,j}(s)=z_{k,j}(t)$ whenever $j\leq k\leq N$. Thus for every $1\leq l\leq L\leq \abs{s}$ we have $$f\left(s_{\lvert L}\right)-f\left(t_{\lvert L}\right)=\sum_{k=l+1}^{L}\sum_{j=0}^k z_{k,j}(s)-z_{k,j}(t),$$ and thus, we get $$\norme{\sum_{k=l+1}^{L}\sum_{j=0}^k z_{k,j}(s)-z_{k,j}(t)}\geq 2(L-l).$$

Moreover, we have \begin{align*} \sum_{k=l+1}^{l+Q^m}\sum_{r=1}^m w_{k,r}(s)-w_{k,r}(t) & =\sum_{k=l+1}^{l+Q^m}\sum_{j=\alpha_{k,m}+1}^{\alpha_{k,0}}z_{k,j}(s)-z_{k,j}(t) \\
&=\sum_{k=l+1}^{l+Q^m}\sum_{j=k-Q^m+1}^kz_{k,j}(s)-z_{k,j}(t)\\
&=\sum_{k=l+1}^{l+Q^m}\sum_{j=0}^kz_{k,j}(s)-z_{k,j}(t)
\end{align*}
because $k-Q^m+1\leq l+1$ whenever $l+1\leq k\leq l+Q^m$. Combining the two facts, we get the desired result.

\end{proof}

Next fix $1\leq m\leq M$, $Q^m\leq l\leq N-Q^m$ and $l+1\leq k\leq N-(Q-1)Q^{m-1}$. For every $0\leq n\leq Q-1$ and for every $t\in T_N$ of length $N$, we have $$\sum_{r=1}^{m-1}w_{k+nQ^{m-1},r}(t)=\sum_{j=k+(n-1)Q^{m-1}+1}^{k+nQ^{m-1}}z_{k+nQ^{m-1},j}(t).$$ In particular, $$\norme{\sum_{r=1}^{m-1}w_{k+nQ^{m-1},r}(t)}\leq \norme{\sum_{j=0}^{k+(n-1)Q^{m-1}}z_{k+nQ^{m-1},j}(t)}+ \norme{\sum_{j=0}^{k+nQ^{m-1}}z_{k+nQ^{m-1},j}(t)}\leq 2c.$$

Now, we apply \ref{estimates_above} to the block functions $\sum_{r=1}^{m-1}w_{k+nQ^{m-1},r}$ with $n$ running from $0$ to $Q-1$. We get \begin{align*} 
\norme{\sum_{n=0}^{Q-1}\sum_{r=1}^{m-1}w_{k+nQ^{m-1},r}(t)-w_{k+nQ^{m-1},r}(s)} &\leq A\left(\sum_{n=0}^{Q-1}\norme{\sum_{r=1}^{m-1}w_{k+nQ^{m-1},r}(t)-w_{k+nQ^{m-1},r}(s)}^p\right)^{\frac{1}{p}}+\eta \\
&\leq 4cAQ^{\frac{1}{p}}+\eta
\end{align*} for all interlacing sequences $s,t\in T_N$ of length $N$.

Thus, assuming that $\eta\leq 4cA$ and summing over $k$, we get \begin{align*}
\norme{\sum_{k=l+1}^{l+Q^m}\sum_{r=1}^{m-1} w_{k,r}(s)-w_{k,r}(t)}= & \norme{\sum_{k=l+1}^{l+Q^{m-1}}\sum_{n=0}^{Q-1}\sum_{r=1}^{m-1} w_{k+nQ^{m-1},r}(s)-w_{k+nQ^{m-1},r}(t)} \\
&\leq  8cAQ^{\frac{1}{p}}Q^{m-1} \\
&\leq Q^{m}
\end{align*} if $Q$ was chosen bigger than $(8cA)^{p^*}$, where $p^*$ is the conjugate exponent of $p$.

Combining this and the lemma, we get that whenever we take interlacing $s,t\in T_N$ of length $N$ satisfying $\abs{a_{s,t}}=l$, we have $$\norme{\sum_{k=l+1}^{l+Q^m}w_{k,m}(s)-w_{k,m}(t)}\geq Q^m,$$ and thus at least one of the quantities $\norme{\sum_{k=l+1}^{l+Q^m}w_{k,m}(s)}$ or $\norme{\sum_{k=l+1}^{l+Q^m}w_{k,m}(t)}$ is bigger than $\frac{1}{2}Q^m$. Then, using Ramsey's theorem, it is easy to get a full subtree of $T_N$ where this inequality holds for every sequence of length $N$.

Consequently, we can assume up to the successive extraction of finitely many full subtrees that for all $ t\in T_N$ of length $N$, for all $1\leq m\leq M$ and for all $Q^m\leq l\leq N-Q^m$, we have $$\norme{\sum_{k=l+1}^{l+Q^m}w_{k,m}(t)}\geq \frac{1}{2}Q^m.$$ Now take $1\leq \gamma\leq Q^{M-m}$ and let $l=\gamma Q^m$. Then $$\sum_{k=l+1}^{l+Q^m}w_{k,m}(t)=\sum_{k=\gamma Q^m+1}^{(\gamma+1) Q^m}w_{k,m}(t).$$ Using the preceding inequality and summing over $\gamma$ we get $$\sum_{\gamma=1}^{Q^{M-m}}\norme{\sum_{k=\gamma Q^m+1}^{(\gamma+1) Q^m}w_{k,m}(t)}\geq \frac{Q^M}{2}.$$ Thus, by triangular inequality, $$\sum_{k=Q^m+1}^{N}\norme{w_{k,m}(t)}\geq  \frac{Q^M}{2}.$$

Finally,  let us recall that $k\geq Q^m$ implies $M_k> m$. Thus, after reordering, we obtain $$\sum_{k=1}^{N}\sum_{m=1}^{M_k}\norme{w_{k,m}(t)}\geq \sum_{m=1}^M\sum_{k=Q^m+1}^{N}\norme{w_{k,m}(t)} \geq M\frac{Q^M}{2}.$$

Gathering the two estimates, we get that if $Q$ is bigger than some constant depending only on $a,A,e,c$ and $p^*$, we have $$M\frac{Q^M}{2} \leq (M+1)^{\frac{1}{q^*}}Q^{M+1}.$$ This gives a contradiction for $M$ large enough.

\end{proof}

\section{The upper $\ell_p$ and lower $\ell_q$ estimates in the trees.}

In this section, we prove the two propositions stated in the last section giving us estimates on the norm of the sum of block functions acting on the upper stages of the trees. Even if there are more quantifiers than usual, the first proposition is not really new and does not require quasi-reflexivity. It is indeed a standard thing now to extract a full subtree with lower $\ell_q$ estimates from a weak$^*$-null tree in a $q-AUC^*$ dual. Similar extractions can be done for upper $\ell_p$ estimates from weak-null trees in asymptotically uniformly smooth spaces, but this cannot be used here because we are looking at weak$^*$ null trees in the bi-dual of a $p-AUS$ space. We give the proof of the first proposition below for completeness and because the structure of the proof, similar to the one of the proof coming after will be easier to get since there are fewer technical arguments.

\begin{proof}[Proof of proposition \ref{estimates_below}]

We will show by induction on $L$ that for all $\xi>0$, there is a full subtree $T$ of $T_N$ such that for all $1\leq L\leq N$ and for all choice of block functions $B_1,\dots, B_L$ as in the statement of the proposition, we have $$\abs{\sum_{l=1}^LB_l(t)}\geq N_L^\delta\Big(\abs{B_1(t)},\dots,\abs{B_L(t)}\Big)-\xi$$ for all $t\in T$ of length $\abs{t}\geq j_L$, where $j_L$ corresponds to the maximal ``height'' of the block function $B_L$. Note that all our blocks $B_l(t)$ are of norm $\abs{.}$ at most $R=2Nc$. 

For $L=1$, the property is satisfied on the whole $T_N$ for all choice of $\xi>0$ because $N_1^\delta=\abs{.}$ by convention.

Now, suppose that our property it is satisfied for all choice of $\xi>0$ for a given $1\leq L\leq N-1$. Fix $\eta>0$. By the uniform continuity of $N_2^\delta$, we can find a  $\nu>0$ such that $\abs{N_2^\delta(u)-N_2^\delta(v)}\leq \frac{\eta}{2}$ whenever $\norme{u-v}_1\leq \nu$ in $\R^2$. For our later use, we assume that $\nu\leq \frac{\eta}{2}$. We may assume that the inequalities for $L$ block functions are satisfied on the whole $T_N$ for the constant $\xi=\nu$. 

\bigbreak

First fix $L$ block functions $B_1,\dots, B_L$ with $j_L\leq N-1$  and fix $t\in T_N$ with $\abs{t}=j_L$. Assuming that $\sum_{l=1}^LB_l(t)\neq 0$, we apply the lemma \ref{lm_w_conv}. There is a weak$^*$-neighborhood $V=V_{x_1^*,\dots, x_m^*;\varepsilon}$ of $0$ such that for all  $x^{**}\in V\cap RB_{(X^{**},\abs{.})}$ we have \begin{align*}
\abs{\sum_{l=1}^LB_l(t)+x^{**}} &\geq \abs{\sum_{l=1}^LB_l(t)}\left(1+\modconvunifw{X^*}{\frac{\abs{x^{**}}}{\abs{\sum_{l=1}^LB_l(t)}}}\right)-\frac{\eta}{2} \\
&= N_2^{\delta}\left(\abs{\sum_{l=1}^LB_l(t)}, x^{**}\right)-\frac{\eta}{2}
\end{align*}
From this, we easily deduce using the inequality for $L$ block functions, the definition of $N_{L+1}^{\delta}$ and our choice of $\nu$ that $$\abs{\sum_{l=1}^LB_l(t)+x^{**}}\geq N_{L+1}^\delta\Big(\abs{B_1(t)},\dots,\abs{B_L(t)},\abs{x^{**}}\Big)-\eta$$ whenever $x^{**}\in V$ and $\abs{x^{**}}\leq R$.

Our goal now is to extract a full subtree over the sequence $t$ which is fully contained in the weak$^*$-neighborhood $V$. We know that $w^*-\lim z_{k,j_L+1}(t\smallfrown n)=0$ for every $N\geq k\geq j_L+1$. Thus we can find some $N_1\geq 1$ such that $$\forall n_1\geq N_1,\ \forall N\geq k\geq j_L +1,\ z_{k,j_L+1}(t\smallfrown n_1)\in V_{x_1^*,\dots, x_m^*;\frac{\varepsilon}{2}}.$$
Then fix some $n_1\geq N_1$. Again, we know that $w^*-\lim z_{k,j_L+2}(t\smallfrown (n_1,n))=0$ for every $N\geq k\geq j_L+2$. Thus we can find some $N_2(n_1)\geq 1$ such that $$\forall n_2\geq N_2(n_1),\ \forall N\geq k\geq j_L +2,\ z_{k,j_L+2}(t\smallfrown (n_1,n_2))\in V_{x_1^*,\dots, x_m^*;\frac{\varepsilon}{4}}.$$
Iterating this procedure, we obtain a full subtree $T^{(t)}$ of $T_{N-j_L}$ such that $$\forall (n_1,\dots, n_j)\in T^{(t)},\ \forall N\geq k\geq j_L+j,\ z_{k,j_L+j}(t\smallfrown (n_1,\dots,n_j))\in V_{x_1^*,\dots, x_m^*;\frac{\varepsilon}{2^j}}.$$

Consequently, this subtree satisfies: for all choice of block function $B_{L+1}$, for all $s\in T^{(t)}$ such that $\abs{t\smallfrown s}\geq j_{L+1}$, $B_{L+1}(t\smallfrown s)\in V$ and thus $$\abs{\sum_{l=1}^LB_l(t)+B_{L+1}(t\smallfrown s)} \geq N_{L+1}\Big(\abs{B_1(t)},\dots,\abs{B_L(t)},\abs{B_{L+1}(t\smallfrown s)}\Big)-\eta.$$
Note that $B_l(t)=B_l(t\smallfrown s)$ for all $1\leq l\leq L$ because the block function $B_l$ only depends on the $j_l\leq j_L$ first coordinates of a sequence. Thus, by ``gluing'' every $T^{(t)}$ over the corresponding point $t$, we get a full subtree $T$ of $T_N$ satisfying the required property for our initial choice of block functions $B_1,\dots, B_L$. 

Since choosing $L$ block functions is equivalent to choosing integers $0\leq i_1\leq j_1<\dots<i_L\leq j_L\leq N-1$ and  $N\geq k_i\geq j_i$ a finite number of successive extractions will give us the desired inequality on a full subtree of $T_N$ for every choice of $L+1$ block functions $B_1,\dots, B_{L+1}$. 

\end{proof}

Now let us give the proof of the second proposition.

\begin{proof}[Proof of proposition \ref{estimates_above}]

Again, we will show by induction on $L$ that for all $\xi>0$, there is a full subtree $T$ of $T_{2N}$ such that for all $1\leq L\leq N$ and for all choice of block functions $B_1,\dots, B_L$, we have $$\norme{\sum_{l=1}^L D_l(s,t)} \leq N_L^{\overline{\rho}_X}\Big(\norme{D_1(s,t)},\dots,\norme{D_L(s,t)}\Big)+\xi$$ whenever $s,t\in T_N$ are of length $\abs{t}=\abs{s}\geq j_L$ and satisfies $s\propto t=(s_1,t_1,s_2,\dots)\in T$, where $D_l(s,t)$ is defined as the difference $B_l(s)-B_l(t)$. Note that these objects are all of norm at most $R=4Nc$.

Again, this is clear for $L=1$. Suppose that the property is satisfied for all choice of $\xi$ for some $1\leq L\leq N-1$. Fix $\eta>0$. By the uniform continuity of $N_2^{\overline{\rho}_X}$, we can find a $\nu>0$ such that $\abs{N_2^{\overline{\rho}_X}(u)-N_2^{\overline{\rho}_X}(v)}\leq \frac{\eta}{4}$ whenever $\norme{u-v}_1\leq \nu$ in $\R^2$. For our later use, we assume that $\nu\leq \frac{\eta}{4}$. Again, we may assume that the inequalities for $L$ differences of block functions are satisfied for every $s,t\in T_N$ of same length for the constant $\nu$.

Since $X$ is quasi-reflexive, there is a space $E$ of finite dimension such that: $X^{**}=X\oplus E$.  For all $u\in T_N$, let $z_{k,j}(u)=x_{k,j}(u)+e_{k,j}(u)$ be the associated decomposition in this sum. Also denote by $D_l^{(X)}$ and $D_l^{(E)}$ the projections of the functions $D_l$ respectively on $X$ and on $E$. By Ramsey's theorem, we may assume after passing to a full subtree that for all $1\leq j\leq k\leq N$ and  for all $u,v\in T_N$ of length $N$ we have $$\norme{e_{k,j}(u)-e_{k,j}(v)}\leq \frac{\nu}{N}.$$ Note that this inequality holds in fact whenever $e_{k,j}(u)$ and $e_{k,j}(v)$ are defined since they only depends on the $j$ first coordinates of $u$ and $v$. 

Now fix $L$ block functions $B_1,\dots, B_L$ and fix $w\in T_{2N},\ \abs{w}= 2j_L$. Also take $s,t\in T_N$ with $\abs{s}=\abs{t}= j_L$ such that $s\propto t=w$.

Again, assuming that $\sum_{l=1}^LD^{(X)}_l(s,t)\neq 0$ and applying lemma \ref{lm_smooth} there is a weak$^*$-neighborhood $V=V_{x_1^*,\dots, x_m^*;\varepsilon}$ of $0$ such that for all $x^{**}\in V\cap RB_{X^{**}}$, we have \begin{align*}
\norme{\sum_{l=1}^LD^{(X)}_l(s,t)+x^{**}} & \leq \norme{\sum_{l=1}^LD^{(X)}_l(s,t)}\left(1+\modlissunif{X}{\frac{\norme{x^{**}}}{\norme{\sum_{l=1}^LD^{(X)}_l(s,t)}}}\right)+\frac{\eta}{4} \\
&=N_2^{\overline{\rho}_X}\left(\norme{\sum_{l=1}^LD^{(X)}_l(s,t)}, \norme{x^{**}}\right)+\frac{\eta}{4}
\end{align*}

Now, we have  $\norme{\sum_{l=1}^LD^{(E)}_l(s,t)}\leq \nu$ thanks to the concentration in $E$ obtained before. Thus \begin{align*}
\norme{\sum_{l=1}^LD_l(s,t)+x^{**}} & \leq N_2^{\overline{\rho}_X}\left(\norme{\sum_{l=1}^LD_l(s,t)}+\nu, \norme{x^{**}}\right)+\nu+\frac{\eta}{4} \\
&\leq N_2^{\overline{\rho}_X}\left(\norme{\sum_{l=1}^LD_l(s,t)}, \norme{x^{**}}\right)+\frac{3\eta}{4} \\
&\leq N_{L+1}^{\overline{\rho}_X}\Big(\norme{D_1(s,t)},\dots,\norme{D_L(s,t)},\norme{x^{**}}\Big)+\eta
\end{align*} whenever $x^{**}\in V$ and $\abs{x^{**}}\leq R$ using the inequality for $L$ difference functions, the definition of $N_{L+1}^{\overline{\rho}_X}$ and the choice of $\nu$.

Using the same arguments as in the proof of the first proposition, we can find a full subtree $T^{(w)}\subset T_{N-j_L}$ such that for all $u=(u_1,\dots, u_j)\in T^{(w)}$ and for all $k\geq j_L+j$, $z_{k,j_L+j}(s\smallfrown u)\in V_{\frac{\varepsilon}{4^j}}$ and $z_{k,j_L+j}(t\smallfrown u)\in V_{\frac{\varepsilon}{4^j}}$.

Thus, for all choice of block function $B_{L+1}$, for all $u,v\in T^{(w)}$ such that $\abs{s\smallfrown u}=\abs{t\smallfrown v}\geq j_{L+1}$, $D_{L+1}(s\smallfrown u,t\smallfrown v)\in V$ and so $$\norme{\sum_{l=1}^LD_l(s,t)+D_{L+1}(s\smallfrown u, t\smallfrown v)} \leq N_{L+1}^{\overline{\rho}_X}\Big(\norme{D_1(s,t)},\dots,\norme{D_L(s,t)},\norme{D_{L+1}(s\smallfrown u, t\smallfrown v)}\Big)+\eta.$$

Noting that the function $D_l$ only depends on the $j_l\leq j_L$ first coordinates of both sequences, and considering the full subtree $T_2^{(w)}$ of $T_{2(N-j_L)}$  for which each sequence of even length is obtained by interlacing two sequences of $T^{(w)}$, we can conclude in the same way as in the preceding proof.

\end{proof}

\section{Application to the non-embeddability into James spaces.}

Let $p\in(1,\infty)$ of conjugate exponent $p^*$. The subspace of $c_0$ composed of all sequences of finite $p$-variation endowed with the $p$-variation norm is called \emph{$p$-James space} and denoted by $\mathcal{J}_p$. Let us recall that the $p$-variation of a sequence $x=(x_n)_{n\geq 1}$ is the quantity $$\norme{x}_{\mathcal{J}_p}^p=\sup_{k\geq 1}\ \sup_{1\leq n_1<\dots<n_k}\ \sum_{i=1}^{k-1} \abs{x_{n_{i+1}}-x_{n_i}}^p.$$

It is a well known fact that $\mathcal{J}_p$ is a quasi-reflexive Banach space admitting an equivalent $p$-AUS norm and such that $J_p^*$ admits an equivalent $p^*$-AUS norm. It is also known that $\mathcal{J}^{**}_p$ admits an equivalent $p$-AUS norm. The reader can consult \cite{LPP} for more information and for references. As a consequence of our main result, we get the following:

\begin{thm}

The family $(T_N)_{N\geq 1}$ does not embed equi-Lipschitz in $\mathcal{J}_p$ and it does not embed equi-Lipschitz in $\mathcal{J}_p^*$.

\end{thm}

\bibliographystyle{plain}
\bibliography{biblio}
\nocite{*}

\end{document}